\documentclass[12pt]{amsart}
\usepackage[colorlinks=true,pagebackref,hyperindex,citecolor=NavyBlue,linkcolor=Mahogany]{hyperref}
\usepackage[dvipsnames]{xcolor}
\usepackage{verbatim}
\usepackage{amsmath}
\usepackage{amsfonts}
\usepackage{amssymb}
\usepackage{color}
\usepackage[all]{xy}  
\usepackage{enumerate}

\usepackage[top=1in, bottom=1in, left=0.9in, right=0.9in]{geometry}
\usepackage{mathrsfs}

\usepackage{stmaryrd}
\usepackage{verbatim}
\theoremstyle{definition}
\newtheorem{theorem}{Theorem}[section]
\newtheorem{theoremx}{Theorem}
\numberwithin{equation}{section}

\newtheorem*{theorem*}{Theorem}

\newtheorem{corollary}[theorem]{Corollary}
\newtheorem{lemma}[theorem]{Lemma}
\newtheorem{proposition}[theorem]{Proposition}

\newtheorem{notation}[theorem]{Notation}

\newtheorem*{claim*}{Claim}

\theoremstyle{definition}
\newtheorem{definition}[theorem]{Definition}

\newtheorem{conjecture}[theorem]{Conjecture}
\newtheorem{remark}[theorem]{Remark}

\newtheoremstyle{TheoremNum}
        {8pt}{8pt}              %%% space between body and theorem
        {\upshape}                      %%% theorem body font
        {}                              %%% Indent amount (empty = no indent)
        {\bfseries}                     %%% theorem head font
        {.}                             %%% Punctuation after theorem head
        {.5em}                             %%% Space after theorem head
        {\thmname{#1}\thmnote{ \bfseries #3}}%%% Thm head spec
  \theoremstyle{TheoremNum}

\newcommand{\NN}{\mathbb{N}}
\newcommand{\ZZ}{\mathbb{Z}}

\newcommand{\HF}{\operatorname{HF}}

\newcommand{\ov}[1]{\overline{#1}}

\renewcommand{\leq}{\leqslant}
\renewcommand{\geq}{\geqslant}

\newcommand{\kk}{\Bbbk}

\newcommand{\out}{\mathfrak{out}}

\renewcommand{\d}{\mathbf{d}}

\newcommand{\EGHH}[2]{{\rm EGH}_{{#1},{#2}}}
\newcommand{\LL}{\mathbb{L}}
\newcommand{\f}{\mathfrak{f}}
\newcommand{\x}{\mathsf{x}}
\newcommand{\g}{\mathfrak{g}}
\newcommand{\lex}{{\rm lex}}

\title[The EGH conjecture for fast-growing degree sequences]{The Eisenbud-Green-Harris conjecture for fast-growing degree sequences}
\author{Giulio Caviglia}
\address{Department of Mathematics, Purdue University, 150 N. University Street, West Lafayette, IN 47907-2067, USA}
\email{gcavigli@purdue.edu}
\author{Alessandro De Stefani}
\address{Dipartimento di Matematica, Universit{\`a} di Genova, Via Dodecaneso 35, 16146 Genova, Italy}
\email{destefani@dima.unige.it}

\subjclass[2010]{13D02, 13A02, 13A15, 13P10}
\keywords{Hilbert function, lexicographic ideal, lex-plus-power ideal, regular sequence}

\begin{document}

\begin{abstract}
Let $S$ be a standard graded polynomial ring over a field, and $I$ be a homogeneous ideal that contains a regular sequence of degrees $d_1,\ldots,d_n$. We prove the Eisenbud-Green-Harris conjecture when the forms of the regular sequence satisfy $d_i \geq \sum_{j=1}^{i-1}(d_j-1)$, improving a result obtained in 2008 by the first author and Maclagan. Except for the sporadic case of a regular sequence of five quadrics, recently proved by G{\"u}nt{\"u}rk{\"u}n and Hochster, the results of this article recover all known cases of the conjecture where only the degrees of the regular sequence are fixed, and include several additional ones. 
\end{abstract} 

\maketitle

\section{Introduction}
This article deals with a conjecture of Eisenbud, Green and Harris which lies at the crossroad of commutative algebra, classical algebraic geometry, and extremal combinatorics. The conjecture originates from interpreting a classical result on points that lie on the intersection of three plane cubics, the Cayley-Bacharach theorem \cite{EGH_CB}, from the perspective of a result from extremal combinatorics, the Kruskal-Katona theorem on the number of faces of a simplicial complex \cite{Kruskal,Katona}, later generalized by Clements and Lindstr{\"o}m \cite{CL}.

%\ale{tolto Macaulay per ora}

%Let $S=\kk[x_1,\ldots,x_n]$ be a standard graded polynomial ring over a field $\kk$. Given a homogeneous ideal $I$, Macaulay's Theorem states that there exists a lexicographic ideal with the same Hilbert function as $I$. A lexicographic ideal is a monomial ideal which, in each degree, is either the zero vector space or it contains exactly all the monomials which are greater or equal than a given one with respect to the lexicographic order. Another way to interpret Macaulay's Theorem is in terms of growth of Hilbert functions: the knowledge of the Hilbert function of $I$ in a given degree $j$ allows to estimate the growth in degree $j+1$. In fact, Macaulay's Theorem states that in degree $j+1$ the ideal $I$ grows at least as a lexicographic ideal which has the same Hilbert function as $I$ in degree $j$.

In algebraic terms, the Eisenbud-Green-Harris conjecture (henceforth EGH, see Conjecture \ref{EGH}) is also an attempt to improve the celebrated theorem of Macaulay, which classifies all the possible Hilbert functions of homogeneous ideals in a polynomial ring over a field, by taking into account additional information about the ideals themselves \cite{EGH}. Namely, the information that an ideal contains a regular sequence of given degrees $d_1,\ldots,d_h$ should, in principle, give more restrictions on the growth of its Hilbert function. In this setup, the natural substitute for the lexicographic ideals studied by Macaulay are the so-called lex-plus-power ideals (henceforth, LPP). As the name suggests, such objects are simply monomial ideals that can be written as a lexicographic ideal plus an ideal generated by pure powers of the variables. %, that is, $(x_1^{d_1},\ldots,x_h^{d_h})$.

%\ale{cambiare} The EGH conjecture was also introduced and studied in relation to Cayley-Bacharach type theorems \cite{EGH_CB}. For instance, see \cite{GK, CDS_CB} for results in this direction.

The EGH conjecture for ideals containing a regular sequence of degrees $d_1 \leq\ldots \leq d_h$, despite being widely open in general, is known to hold in a few specific cases, which can be divided in two categories. The first category includes regular sequences or ideals of special type; for instance, the case when the regular sequence is monomial \cite{CL,MP,CK1}, when the ideal is generated by a quadratic regular sequence plus general quadratic forms \cite{HP,Gash}, when the regular sequence decomposes as a product of linear forms \cite{A}, or when the ideal is monomial (for quadrics see \cite{CCV}, in general see \cite{A1}). The second category includes cases in which only the degrees of the regular sequence and its length are fixed: the case $h=2$ \cite{Richert,Cooper}, the case $h=3$, provided $d_1=2$ or $d_1=3$ and $d_2=d_3$ \cite{Cooper1}, and the case $d_1=\ldots=d_h=2$ and $h \leq 5$ \cite{GuHo}. In the second category falls also a result due to the first author and Maclagan from 2008, where they prove that the EGH conjecture holds if the degrees satisfy $d_i > \sum_{j=1}^{i-1}(d_j-1)$ for all $i \geq 3$,  see \cite[Theorem 2]{CM}. %

%We extend this result by improving each inequality by $1$.

With the exception of the case of five quadrics due to G{\"u}nt{\"u}rk{\"u}n and Hochster \cite{GuHo}, our main result, together with Proposition \ref{Prop Cooper}, covers all known cases which fall in the second category, and includes several more.

\begin{theoremx}[see Corollary \ref{Cor socle}] \label{THMX A}Let $I \subseteq S=\kk[x_1,\ldots,x_n]$ be a homogeneous ideal containing a regular sequence of degrees $d_1\leq \ldots \leq d_h$, such that $d_i \geq \sum_{j=1}^{i-1} (d_j-1)$ for all $i \geq 3$. Then $I$ satisfies the EGH conjecture, that is, there exists a lex-plus-power ideal containing $(x_1^{d_1},\ldots,x_h^{d_h})$ with the same Hilbert function as $I$.
\end{theoremx}

We actually prove a stronger statement than Theorem \ref{THMX A}: if any homogeneous ideal $I$ containing a given regular sequence $f_1,\ldots,f_{h-1}$ of degrees $d_1 \leq \ldots \leq d_{h-1}$ satisfies the EGH conjecture with respect to such a sequence, then $I+(f_h)$ satisfies the EGH conjecture with respect to the sequence $f_1,\ldots,f_{h-1},f_h$ for any element $f_h$ which is regular modulo $(f_1,\ldots,f_{h-1})$, and has degree at least $\sum_{i=1}^{h-1}(d_i-1)$ (see Theorem \ref{THM socle}).

We point out that the strategy employed in \cite{CM} could not be directly used to prove Theorem \ref{THMX A}. In fact, \cite[Theorem 2]{CM} relies on estimates on the Hilbert function of ideals containing shorter regular sequences, which are then glued together using linkage. In Theorem \ref{THMX A}, on the other hand, if $d_i=\sum_{j=1}^{i-1} (d_j-1)$ for some $i$ there is one critical degree, namely $d_i$, where the gluing process could in principle go wrong. We tackle this problem by carefully estimating the growth of the Hilbert function in that degree, and comparing it with the one of the corresponding LPP ideal. Using the same techniques, in Proposition \ref{Prop Cooper} we also prove some additional results which are not covered by the main theorem. %For instance, with the notation used above, we recover the case $h=3$, $d_1=3$ and $d_2=d_3$ proved in \cite{Cooper1}, and we prove the new case $h=4$, $d_1=d_2=2$ and $d_3=d_4$ (see Proposition \ref{Prop Cooper}).

As a final remark, consider the following very concrete scenario. Suppose we are given an ideal $I \subseteq \kk[x_1,x_2,x_3,x_4]$, with the knowledge that $I$ contains a regular sequence of degrees $(4,5,6)$, and that $\HF(I;6) = 20$. While the EGH conjecture is not known to hold in this case, after possibly enlarging the field $\kk$ (which does not affect our considerations) the ideal $I$ will certainly contain a regular sequence of degrees $(4,5,7)$. Then we can apply Theorem \ref{THMX A} to obtain that $\HF(I;7) \geq 41$. While this is not the sharpest estimate predicted by the EGH conjecture (that is, $\HF(I;7) \geq 43$), it still provides a significantly better estimate than the one coming from Macaulay's Theorem (that is, $\HF(I;7) \geq 35$). 

The strategy outlined in the example above illustrates that %, while all the listed results on the EGH conjecture can only be applied when some specific conditions are met by the ideal or by the degrees of the regular sequence, 
the main result of \cite{CM} and our improvement, Theorem \ref{THMX A}, provide an estimate on the growth of the Hilbert function, which is more accurate than Macaulay's Theorem, for {\it any} homogeneous ideal.

\section{Preliminaries}
Let $\kk$ be a field, and $S=\kk[x_1,\ldots,x_n]$ be a polynomial ring over $\kk$, with the standard grading. Given a finitely generated graded $S$-module $M = \bigoplus_{j \in \ZZ} M_j$, we will denote by $\HF(M;j) = \dim_\kk M_j$ its Hilbert function in degree $j$.

On $S$ we will consider the degree-lexicographic order, which we denote $>_{\lex}$, and the variables of $S$ will be ordered by $x_1 >_{\lex} x_2 >_{\lex} \ldots >_{\lex} x_n$. A graded vector space $V$ is a lex-segment if, whenever $u,v \in S$ are monomials of the same degree with $u >_{\lex} v$,  then $u \in V$ whenever $v \in V$. A monomial ideal $L \subseteq S$ is called a lexicographic ideal if its graded components $L_j$ are lex-segments for all $j \in \ZZ$.

A degree sequence is a vector $\d=(d_1,\ldots,d_h) \in \NN^h$, with $1 \leq d_1 \leq \ldots \leq d_h$. Given a regular sequence $f_1,\ldots,f_h \in S$ we say that it has degree $\d = (d_1,\ldots,d_h)$ if $\deg(f_i)=d_i$ for $i=1,\ldots,h$.

\begin{definition} Let $S=\kk[x_1,\ldots,x_n]$, and $\d$ be a degree sequence. A monomial ideal $\LL$ is called a $\d$-LPP ideal if  $\LL = (\x^\d) + L$, where $L$ is a lexicographic ideal.
\end{definition}

\begin{remark} \label{Remark oplus}
If $\LL$ is a $\d$-LPP ideal then we may also write $\LL=(\x^\d) \oplus V$, where $V=\bigoplus_j V_j$ is a graded $\kk$-vector space generated by monomials. Observe that, in general, $V$ is not an ideal.
\end{remark}

We now recall the most current version of the Eisenbud-Green-Harris conjecture (see \cite{CM}).
\begin{conjecture}[$\EGHH{\d}{n}$] \label{EGH} Let $I \subseteq S = \kk[x_1,\ldots,x_n]$ be a homogeneous ideal containing a regular sequence of degree $\d$. We say that $I$ satisfies $\EGHH{\d}{n}$ if there exists a $\d$-LPP ideal with the same Hilbert function as $I$.
\end{conjecture}

Observe that the case in which no regular sequence is taken into account is the well-known Macaulay's Theorem on the existence of a lexicographic ideal with the same Hilbert function as a given homogeneous ideal. As in Macaulay's Theorem, the $\d$-LPP ideal of Conjecture \ref{EGH} is unique, whenever it exists.

We make the following definition, based on \cite[Definition 11]{CM}. 

\begin{definition} Let $j$ be a non-negative integer, and $I \subseteq S = \kk[x_1,\ldots,x_n]$ be a homogeneous ideal containing a regular sequence of degree $\d$. We say that $I$ satisfies $\EGHH{\d}{n}(j)$ if there exists a $\d$-LPP ideal $\LL$ such that $\HF(I;j) = \HF(\LL;j)$ and $\HF(I;j+1) \geq \HF(\LL;j+1)$.
\end{definition}

It is immediate to see that an ideal $I$ satisfies $\EGHH{\d}{n}$ if and only if it satisfies $\EGHH{\d}{n}(j)$ for all non-negative integers $j$. 
\section{Main result}

We start by setting up some notation. Let $\f = (f_1,\ldots,f_h)$ be an ideal of $S=\kk[x_1,\ldots,x_n]$ generated by a regular sequence of degrees $\d=(d_1,\ldots,d_h)$, with $1 \leq d_1 \leq \ldots \leq d_h$. After possibly enlarging $\kk$, which does not affect any of our considerations, we may find $n-h$ linear forms $\ell_{1},\ldots,\ell_{n-h}$ such that $f_1,\ldots,f_h,\ell_1,\ldots,\ell_{n-h}$ is a maximal regular sequence in $S$. After a change of coordinates, we may assume that $\ell_i = x_{h+i}$ for all $i=1,\ldots,n-h$, and therefore we may view $\f+(x_{h+1},\ldots,x_n)$ as an ideal $\ov{\f}$ inside $\ov{S}$, still generated by a regular sequence of degrees $\d$. We will refer to $\ov{\f}$ as an Artinian reduction of $\f$. Clearly, the Artinian reduction depends on the choice of the original linear forms $\ell_{1},\ldots,\ell_{n-h}$. However, by \cite[Theorem 4.1]{CK} we have that if every homogeneous ideal of $\ov{S}$ that contains $\ov{\f}$ satisfies $\EGHH{\d}{h}$, then every homogeneous ideal of $S$ that contains $\f$ satisfies $\EGHH{\d}{n}$.
\begin{remark} To the best of our knowledge, in all the cases where the EGH conjecture is known to hold the converse to the last statement is also true. For instance, most of the known cases only require numerical conditions on the degree sequence $\d$ to be satisfied, independently of whether $\f$ is Artinian or not.
\end{remark}

Before proving the main theorem, we need some preparatory results. 

\begin{lemma} \label{Lemma lex} Let $\d=(d_1,\ldots,d_h)$ be a degree sequence, and $\LL$ be a $\d$-LPP ideal such that $\HF(\LL;D)>\HF((\x^\d);D)$ for some integer $D$. Given $D'$ such that $D \leq D' \leq \sum_{i=1}^h (d_i-1)$, there exists a $\d$-LPP ideal $\LL'$ such that 
\[
\HF(\LL';j) = \begin{cases}  \HF((\x^\d);j) & \text{ if } j < D \\ \HF(\LL;j)-1 & \text{ if } D \leq j \leq D' \\
\HF(\LL;j) & \text{ otherwise. }%\giulio{PUNTO?}}
\end{cases}
\]
\end{lemma}
\begin{proof}
As in Remark \ref{Remark oplus}, we may write $\LL=(\x^\d) \oplus V$, where $V=\bigoplus_{j} V_j$ is a graded vector space generated by monomials. By assumption $V_D \ne 0$, and this implies $V_j \ne 0$ for all $D \leq j \leq \sum_{i=1}^h (d_i-1)$. In particular, $V_j \ne 0$ for $D \leq j \leq D'$. If we let $\ov{V}_j$ be the $\kk$-vector space generated by all monomials of $V_j$ except the smallest with respect to the lexicographic order, then $W = \left(\bigoplus_{D \leq j \leq D'} \ov{V}_j \right) \oplus \left(\bigoplus_{j > D'} V_j \right)$ is such that $\LL' = (\x^\d) \oplus W$ is still an ideal. In fact, for $D < j \leq D'$, if the smallest monomial $v \in V_j$ is divisible by a monomial $u \in V_{j-1}$, then $u$ must necessarily be the smallest monomial of $V_{j-1}$ with respect to the lexicographic order, by standard properties of lex-segments. The fact that $\LL'$ is a $\d$-LPP ideal and it satisfies the desired conditions on the Hilbert function are now trivial to check.
\end{proof}
The following lemma is a direct consequence of Gotzmann's Persistence Theorem \cite{G,Green_Gotzmann}. We single this result out, since it will be used as stated in the proof of the main theorem.
\begin{lemma} \label{Lemma Gotzmann} Let $I \subseteq S=\kk[x_1,\ldots,x_n]$ be a homogeneous ideal generated in degree at most $D$. If $0<\HF(S/I;D) = \HF(S/I;D+1) \leq D$, then $\dim(S/I) = 1$. 
\end{lemma}
\begin{proof}
The condition $\HF(S/I;D) = \HF(S/I;D+1) \leq D$ implies, by Macaulay's Theorem, that the lexicographic ideal with the same Hilbert function as $I$ has no minimal generator in degree $D+1$. It follows from Gotzmann's Persistence Theorem \cite{G,Green_Gotzmann} that $\HF(S/I;D) = \HF(S/I;j)$ for all $j \geq D$. As this value is positive by assumption, this forces $\dim(S/I)=1$.
\end{proof}

The next proposition, even though not stated in this generality, is a direct consequence of the techniques used in \cite{CM}.

\begin{proposition} \label{CM improved} Let $I \subseteq S=\kk[x_1,\ldots,x_n]$ be a homogeneous ideal containing an ideal $\f$ generated by a regular sequence $f_1,\ldots,f_h$ of degrees $\d=(d_1,\ldots,d_h)$. Let $\f'=(f_1,\ldots,f_{h-1})$, $\d'=(d_1,\ldots,d_{h-1})$, and fix an Artinian reduction $\ov{\f'}$ of $\f'$ inside $\ov{S}=\kk[x_1,\ldots,x_{h-1}]$. Assume that every homogeneous ideal of $\ov{S}$ containing $\ov{\f'}$ satisfies $\EGHH{\d'}{h-1}$. If $d_h > \sum_{i=1}^{h-1}(d_i-1)$, then $I$ satisfies $\EGHH{\d}{n}$.
\end{proposition}
\begin{proof}
By \cite[Theorem 4.1]{CK}, in order to prove the proposition we may assume that $h=n$, and that the image of $\f'+(x_n)$ inside $S/(x_n)$ is identified with the fixed Artinian reduction $\ov{\f'}$ of $\f'$. Again by \cite[Theorem 4.1]{CK}, we have that any ideal containing $\f'$ satisfies $\EGHH{\d'}{n}$, since any ideal containing $\ov{\f'}$ satisfies $\EGHH{\d}{n-1}$ by assumption. By \cite[Lemma 12]{CM}, it suffices to show that $I$ satisfies $\EGHH{\d}{n}(j)$ for all $0 \leq j \leq d_n-2$. Since $f_n$ has degree $d_n$ it is clear that, for $j$ in this range, $I$ satisfies $\EGHH{\d}{n}(j)$ if and only if $I$ satisfies $\EGHH{\d'}{n}(j)$, and therefore the proof is complete.
\end{proof}

\begin{notation} Let $\LL$ be a $\d$-LPP ideal, and $j \geq 0$ be an integer. If $\LL_j \ne S_j$ we set $\out(\LL;j)$ to be the largest monomial with respect to the lexicograpich order that does not belong to $\LL_j$. If $\LL_j = S_j$, then we set $\out(\LL;j) = 0$.
\end{notation} 

\begin{theorem} \label{THM socle} Let $I \subseteq S=\kk[x_1,\ldots,x_n]$ be a homogeneous ideal containing an ideal $\f$ generated by a regular sequence $f_1,\ldots,f_h$ of degrees $\d=(d_1,\ldots,d_h)$. Let $\f'=(f_1,\ldots,f_{h-1})$, $\d'=(d_1,\ldots,d_{h-1})$, and fix an Artinian reduction $\ov{\f'}$ of $\f'$ inside $\ov{S}=\kk[x_1,\ldots,x_{h-1}]$. Assume that every homogeneous ideal of $\ov{S}$ containing $\ov{\f'}$ satisfies $\EGHH{\d'}{h-1}$. If $d_h \geq \sum_{i=1}^{h-1}(d_i-1)$, then $I$ satisfies $\EGHH{\d}{n}$.
\end{theorem} 
\begin{proof}
Without loss of generality, we may assume that $\kk$ is infinite. By \cite[Theorem 4.1]{CK} we may assume that $h=n$. By Proposition \ref{CM improved}, we may assume that $d_n = \sum_{i=1}^{n-1}(d_i-1)$, so that $\sum_{i=1}^{n} (d_i-1) = 2d_n-1$. By \cite[Lemma 12]{CM}, in order to prove the theorem it suffices to show that $I$ satisfies $\EGHH{\d}{n}(j)$ for all $0 \leq j \leq d_n-1$. As in the proof of Proposition \ref{CM improved} we have that $I$ satisfies $\EGHH{\d'}{n}$ by \cite[Theorem 4.1]{CK}, given that $I$ contains $\f'$ and $\EGHH{\d'}{n-1}$ is assumed to be true. Since for $j \leq d_n-2$ $\EGHH{\d'}{n}(j)$ is clearly equivalent to $\EGHH{\d}{n}(j)$, it suffices to prove that $I$ satisfies $\EGHH{\d}{n}(d_n-1)$. 

Let $\{v_1,\ldots,v_c\} \subseteq I_{d_n-1}$ be the pre-image in $S$ of a $\kk$-basis of $(I/\f)_{d_n-1}$, and consider the ideal $Q= (f_1,\ldots,f_{n-1},v_1,\ldots,v_c)$. First, assume that $f_n \notin Q$. In this case, we have $\HF(I;d_n) \geq \HF(Q;d_n)+1$. By assumption, we have that $Q$ satisfies $\EGHH{\d'}{n}(d_n-1)$, that is, there exists a $\d'$-LPP ideal $\LL$ with the same Hilbert function as $Q$ (and thus as $I$) in degree $d_n-1$, and such that $\HF(Q;d_n) \geq \HF(\LL;d_n)$. If $x_n^{d_n} \in \LL$, then $I_{d_n} = Q_{d_n} = S_{d_n}$, which is a contradiction since $f_n \notin Q$. Therefore $x_n^{d_n} \notin \LL$, and we have that $\HF(\LL+(x_n^{d_n});d_n) = \HF(\LL;d_n)+1$. In particular, we have that $\HF(I;d_n) \geq \HF(\LL+(x_n^{d_n});d_n)$. Since $\LL+(x_n^{d_n})$ is a $\d$-LPP ideal with the same Hilbert function as $I$ in degree $d_n-1$, it follows that $I$ satisfies $\EGHH{\d}{n}(d_n-1)$, as desired.

From now on, assume that $f_n \in Q$. Under this assumption, in order to prove the theorem we may replace $I$ by $Q$. Then, since $Q$ has height $n$ and is generated in degree at most $d_n-1$, we may assume that $f_1,\ldots,f_{n-1},v_c$ is a maximal regular sequence of degree $\d'' = (d_1,\ldots,d_{n-1},d_n-1)$ inside $Q$.  Let $\g=(f_1,\ldots,f_{n-1},v_c)$ and $J=\g:Q$. As the degree of the socle of $S/\g$ is $2d_n-2,$ by linkage (for instance, see \cite[Corollary 5.2.19]{Migliore}), we have that
\[
\HF(S/Q;d_n) = \HF(S/\g;d_n-2)- \HF(S/J; d_n-2), \text{ and }
\]
\[
\HF(S/Q;d_n-1) = \HF(S/\g;d_n-1) - \HF(S/J; d_n-1).
\]
Let $\LL_1$ be the $\d'$-LPP ideal with the same Hilbert function as $J$, which exists by hypothesis. Observe that $\HF(\LL_1;d_n-2) \geq \HF((\x^{\d'});d_n-2) = \HF(\g;d_n-2)$. If equality holds, then by the linkage formula used above we conclude that $\HF(S/Q;d_n) = 0$, that is, $Q_{d_n} = S_{d_n}$. In this case, $Q$ trivially satisfies $\EGHH{\d}{n}(d_n-1)$. If $\HF(\LL_1;d_n-2) > \HF((\x^{\d'});d_n-2)$, then by Lemma \ref{Lemma lex} applied to $\LL_1$ with $D=d_n-2$ and $D'=d_n-1$, there exists a $\d'$-LPP ideal $\LL_2$ such that
\[
\HF(\LL_2;j) = \begin{cases} \HF((\x^{\d'});j)& \text{ if } j < d_n-2 \\
\HF(\LL_1;j)-1 & \text{ if } j=d_n-2,d_n-1 \\
\HF(\LL_1;j) & \text{ if } j > d_n-1
\end{cases}
\]
In particular, $\LL_2+(x_n^{d_n-1})$ is a $\d''$-LPP ideal which has the same Hilbert function as $\LL_1$ in degree $d_n-1$, and value one less than $\LL_1$ in degree $d_n-2$. If we let $\LL_3=(\x^{\d''}):(\LL_2+(x_n^{d_n-1}))$, again by linkage we have $\HF(Q;d_n-1) = \HF(\LL_3;d_n-1)$ and $\HF(Q;d_n) = \HF(\LL_3;d_n)-1$. Furthermore, by \cite[Theorem 1.2]{MP} there exists a $\d''$-LPP ideal $\LL_4$ with the same Hilbert function as $\LL_3$. As in Remark \ref{Remark oplus}, we may write $\LL_4 = (\x^{\d''}) \oplus  V$, where $V = \bigoplus_j V_j$ is a graded $\kk$-vector space. We now want to trade $x_n^{d_n-1} \in (\x^{\d''})$ for $x_n^{d_n} \in (\x^{\d})$ in the ideal $\LL_4$ we have just defined. To do so, we need to adjust the vector space $V_{d_n-1}$ by adding one specific monomial of degree $d_n-1$, and then estimate its growth. 

More specifically, consider the $\kk$-vector space $W=V_{d_n-1} \oplus V_{d_n}$, and let $\LL_5$ be the ideal generated by $(\x^{\d'}) \oplus W$. It can be shown that $\LL_5$ is a $\d'$-LPP ideal.  Let $u = \out(\LL_5;d_n-1)$, and observe that $u \ne 0$. In fact, if $u=0$, then necessarily $x_n^{d_n-1} \in W$, and therefore $(\LL_5)_{d_n-1}=S_{d_n-1}$. This forces $(\LL_4)_{d_n-1}= S_{d_n-1}$, and thus $Q_{d_n-1}=  S_{d_n-1}$, because $Q$ and $\LL_4$ have the same Hilbert function in degree $d_n-1$. But this is a contradiction, since this would imply $Q_{d_n} = (\LL_4)_{d_n} = S_{d_n}$, while we are assuming that $\HF(Q;d_n) = \HF(\LL_4;d_n)-1$.  

Now we let $\LL(\d) = \LL_5 + (u,x_n^{d_n})$, which is a $\d$-LPP ideal. Observe that $x_n^{d_n}$ is necessarily a minimal generator of $\LL(\d)$, by what we have observed above. Moreover, we have that
\begin{align*}
\HF(\LL(\d);d_n-1) & = \HF(\LL_5+(u);d_n-1) =\HF(\LL_5+(x_n^{d_n-1});d_n-1) \\
& =  \HF(\LL_4;d_n-1) = \HF(Q;d_n-1), \text{ and }
\end{align*}
\begin{align*}
\HF(\LL(\d);d_n) & =\HF(\LL_5;d_n) +  \HF((u,x_n^{d_n},\LL_5)/\LL_5;d_n)  \\
& =\HF(\LL_4;d_n) - \HF(\LL_4/\LL_5;d_n) +  \HF((u,x_n^{d_n},\LL_5)/\LL_5;d_n)   \\
& = \HF(Q;d_n)+1 - \HF((x_n^{d_n-1},\LL_5)/\LL_5;d_n) + \HF((u,x_n^{d_n},\LL_5)/\LL_5;d_n).
\end{align*}
To conclude the proof, we want to show that $\HF(\LL(\d);d_n) \leq \HF(Q;d_n)$. Since the degree of $u$ is $d_n-1 = \sum_{i=1}^{n-1}(d_i-1)-1$, it is easy to see from the fact that lex-segments are strongly stable that $\HF((u,x_n^{d_n},\LL_5)/\LL_5; d_n) = 3$ if and only if $u=x_1^{d_1-1} \cdots x_{n-1}^{d_n-2}$, and $\HF((u,x_n^{d_n},\LL_5)/\LL_5; d_n) \leq 2$ in all the other cases.  However, if $u=x_1^{d_1-1} \cdots x_{n-1}^{d_n-2}$, then it follows by a direct computation that $\HF(\LL(\d);d_n-1) = \HF((\x^\d);d_n-1) + 1$. 
In our running assumptions, this forces $c=1$, and $Q=(f_1,\ldots,f_{n-1},v_1)$ to be generated by a regular sequence of degree $\d''$. Since by \cite[Theorem 1.2]{MP} there is a $\d$-LPP ideal with the same Hilbert function as $(\x^{\d''})$, this concludes the proof in this case. 

We will henceforth assume that $\HF((u,x_n^{d_n},\LL_5)/\LL_5; d_n) \leq 2$, so that by the above computation
\[
\HF(\LL(\d);d_n) \leq  \HF(Q;d_n) - \HF((x_n^{d_n-1},\LL_5)/\LL_5;d_n) + 3.
\]

We now need to estimate the growth of $(x_n^{d_n-1},\LL_5)/\LL_5$ in degree $d_n$. In order to do so, we let $\ell = \min\{i =1,\ldots,n \mid x_i x_n^{d_n-1} \notin \LL_5\}$. If $\ell \leq n-2$, then $\HF((x_n^{d_n-1},\LL_5)/\LL_5;d_n) \geq 3$, and the proof is complete. 

For the rest of the proof, assume that $\ell \geq n-1$. Under this assumption, we necessarily have that $u \in \kk[x_{n-1},x_n]$ and, in fact, $u \leq_{\lex} x_{n-1}^{d_{n-1}-1} x_n^{d_n-d_{n-1}}$. Moreover, a direct computation shows that $\HF(S/\LL(\d);d_n-1) \leq d_n-1$. Since $\HF(S/Q;d_n-1) = \HF(S/\LL(\d);d_n-1)$, Macaulay's Theorem implies that $\HF(S/Q;j) \leq \HF(S/Q;d_n-1)$ for all $j \geq d_n-1$. If $\HF(S/Q;d_n-1) = \HF(S/Q;d_n)$, then by Lemma \ref{Lemma Gotzmann} we would have $\dim(S/Q)=1$. However, $\f \subseteq Q$ by our running assumptions, and therefore $\dim(S/Q)=0$, a contradiction. Thus, we have that $\HF(S/Q;d_n) < \HF(S/Q;d_n-1)$. Therefore $\HF(S/Q;d_n) \leq \HF(S/Q;d_n-1) - 1 = \HF(S/\LL(\d);d_n-1) -1 = \HF(S/\LL(\d);d_n)$, where the last equality follows from a direct computation, thanks to the knowledge of the monomial $u$.
\end{proof}

As a consequence of Theorem \ref{THM socle}, we obtain the main result of this article.

\begin{corollary} \label{Cor socle} Let $I \subseteq S=\kk[x_1,\ldots,x_n]$ be a homogeneous ideal containing a regular sequence of degrees $\d$, such that $d_i \geq \sum_{j=1}^{i-1} (d_j-1)$ for all $i \geq 3$. Then $I$ satisfies $\EGHH{\d}{n}$.
\end{corollary}
\begin{proof} This follows immediately by induction on the length $h$ of the regular sequence, and Theorem \ref{THM socle}.
\end{proof}

Consider the EGH conjecture in the case of the degree sequences $\d=(2,d,d)$ and $\d=(3,d,d)$, both covered in \cite{Cooper1}. While the first case is a consequence of Corollary \ref{Cor socle}, the same is not true for $\d=(3,d,d)$. However, by means of the same techniques used in Theorem \ref{THM socle}, we can recover the second case as well, and prove an additional one.

\begin{proposition} \label{Prop Cooper} Let $I \subseteq \kk[x_1,\ldots,x_n]$ be a homogeneous ideal that contains a regular sequence of degree $\d$, where $\d$ is either $(3,d,d)$ or $(2,2,d,d)$. Then $I$ satisfies $\EGHH{\d}{n}$.
\end{proposition} 
\begin{proof} 
Let $\f$ be the ideal generated by a regular sequence of degree $\d$ inside $I$, and $\f' \subseteq \f$ be the ideal generated only by the forms of degree $\d'$, where $\d'=(3,d)$ if $\d=(3,d,d)$, and $\d'=(2,2,d)$ if $\d=(2,2,d,d)$. Observe that Conjecture $\EGHH{\d'}{n}$ is known to hold for every ideal containing a regular sequence of degree $\d'$ by \cite{Richert,Cooper} in the first case, and by Corollary \ref{Cor socle} in the second. By \cite[Lemma 12]{CM}, it suffices to prove that $I$ satisfies $\EGHH{\d}{n}(j)$ for $0 \leq j \leq d-1$. Moreover, as in the proof of Theorem \ref{THM socle}, the critical case is $j=d-1$. Let $\{v_1,\ldots,v_c\}$ be the pre-image in $S$ of a $\kk$-basis of $(I/\f)_{d-1}$, and $Q=\f'+(v_1,\ldots,v_c)$.

First assume $\d=(3,d,d)$. If $f_3 \notin Q$, then since $Q$ satisfies $\EGHH{\d'}{n}$ there exists a $\d'$-LPP ideal $\LL$ with the same Hilbert function as $Q$ in degree $d-1$, and such that $\HF(Q;d) \geq \HF(\LL;d)$. As in the proof of Theorem \ref{THM socle}, one can check that $\LL+(x_3^d)$ is a $\d$-LPP ideal, and we have that $\HF(I;d) \geq \HF(Q;d)+1 \geq \HF(\LL;d)+1=\HF(\LL+(x_3^d);d)$. On the other hand, if $f_3 \in Q$, then we may assume that $f_1,f_2,v_c$ is a regular sequence of degrees $3,d$ and $d-1$. If $d=3$, then we let $\d''=(2,3,3)$, while if $d>3$, then we let $\d''=(3,d-1,d)$. Either way, $I$ satisfies $\EGHH{\d''}{n}$ by Corollary \ref{Cor socle}, and therefore there exists a $\d''$-LPP ideal $\LL$ with the same Hilbert function as $I$. In particular, since $\LL$ contains $(x_1^3,x_2^d,x_3^d)$, we have that $I$ satisfies $\EGHH{\d}{n}$ by \cite[Theorem 1.2]{MP}. 

The case $\d=(2,2,d,d)$ is similar. If $f_4 \notin Q$, then the proof goes as in the previous case, using that $I$ satisfies $\EGHH{\d'}{n}$. If $f_4 \in Q$, then we may assume that $f_1,f_2,f_3,v_c$ forms a regular sequence of degrees $2,2,d$ and $d-1$. If $d=2$ then we let $\d''=(1,2,2,2)$, while if $d>2$ we let $\d''=(2,2,d-1,d)$. Either way, $I$ satisfies $\EGHH{\d''}{n}$ by Corollary \ref{Cor socle}, and we conclude using \cite[Theorem 1.2]{MP} as above that it satisfies $\EGHH{\d}{n}$ as well.
\end{proof}

\end{document}